\documentclass{amsart} 

\usepackage{amsmath,amsfonts,amssymb,mathrsfs,amsthm}

\usepackage[dvipsnames]{xcolor}

\usepackage{tikz}
\usepackage{datetime}

\newtheorem{theorem}{Theorem}[section]
\newtheorem{lemma}[theorem]{Lemma}
\newtheorem{proposition}[theorem]{Proposition}
\newtheorem{corollary}[theorem]{Corollary}

\theoremstyle{definition}
\newtheorem{definition}[theorem]{Definition}

\theoremstyle{remark}

\numberwithin{equation}{section}

\newcommand{\g}{\mathfrak{g}}

\newcommand{\B}{\mathcal{B}}

\newcommand{\E}{\mathcal{E}}
\newcommand{\F}{\mathcal{F}}

\newcommand{\wt}{\mathrm{wt}}

\newcommand\Hom{\mathrm{Hom}}

\usepackage{mathtools}
\DeclarePairedDelimiter{\ceiling}{\lceil}{\rceil}
\DeclarePairedDelimiter{\floor}{\lfloor}{\rfloor}

\begin{document}


\title[Extremal tensor products]{Extremal tensor products of Demazure crystals}  

\author[S. Assaf]{Sami Assaf}
\address{Department of Mathematics, University of Southern California}
\email{shassaf@usc.edu}
\thanks{S.A. supported by Simons Award 953878.}

\author[A. Dranowski]{Anne Dranowski}
\address{Department of Mathematics, University of Southern California}
\email{dranowsk@usc.edu}

\author[N. Gonz\'alez]{Nicolle Gonz\'alez}
\address{Department of Mathematics,  University of California, Berkeley}
\email{nicolle@math.berkeley.edu}

\subjclass[2020]{Primary 05E10; Secondary 05E16, 20G42}



\keywords{Crystal bases, Demazure crystals, Demazure modules, extremal crystals, tensor products, excellent filtrations.}

\begin{abstract}
Demazure crystals are subcrystals of highest weight irreducible $\g$-crystals. In this article, we study tensor products of a larger class of subcrystals, called extremal, and give a local characterization for exactly when the tensor product of Demazure crystals is extremal. We then show that tensor products of Demazure crystals decompose into direct sums of Demazure crystals if and only if the tensor product is extremal, thus providing a sufficient and necessary local criterion for when the tensor product of Demazure crystals is itself Demazure. As an application, we show that the primary component in the tensor square of any Demazure crystal is always Demazure. 
\end{abstract}

\maketitle

%
\section{Introduction}
%
\label{sec:introduction}

For $G$ a connected, simply-connected, semi-simple Lie group, the finite dimensional simple $G$-modules $V(\lambda)$ are indexed by $\lambda\in P^+$ the set of dominant weights.
For $w$ an element of the Weyl group $W$ and $B \subset G$ a Borel subgroup, the \emph{Demazure module} $V_w(\lambda)$ is the $B$-submodule of $V(\lambda)$ generated by the one-dimensional extremal weight space of $V_w(\lambda)$ with weight $w \lambda$ \cite{Dem74}.

Mathieu \cite{Mat89} proved a conjecture of Polo \cite{Pol89} stating that if one twists a Demazure module by an anti-dominant character, then the resulting $B$-module can be filtered with successive quotients given by Demazure modules. 
Such a filtration is called an \emph{excellent filtration}. Excellent filtrations for $B$-modules generalize the notion of \emph{good filtrations} for $G$-modules, where successive quotients are Weyl modules. We identify simple modules and their Weyl modules. The tensor product of two $G$-modules with good filtrations has a good filtration, but having excellent filtrations is not, in general, preserved by tensor products for $B$-modules \cite{vdK89}.

Kashiwara \cite{Kas91} introduced crystal bases as set-theoretic abstractions of Lusztig's geometric canonical bases \cite{Lus90} of representations of quantized universal enveloping algebras, at $q = 0$.
In particular, Weyl modules admit crystals bases and we denote the crystal of $V(\lambda)$ by $\B(\lambda)$. 

Crystals carry some of the structure of their modules. For example, the character of a module can be computed from its crystal.
Moreover, there is a simple combinatorial rule for defining the tensor product crystal $\B(\lambda) \otimes \B(\mu)$ as a disjoint union of crystals which exactly matches the representations appearing in the decomposition of $V(\lambda) \otimes V(\mu)$. 
Note this gives a crystal-theoretic proof that tensor products of modules with good filtrations once again have good filtrations.

The Demazure character formula \cite{Dem74a} was proved by Joseph \cite{Jos85} and generalized by Ishii \cite{Ish15}. For $w\in W$, Littelmann \cite{Lit95} conjectured, and proved for classical types, the existence of a subset $\B_w(\lambda) \subseteq \B(\lambda)$ whose character corresponds to that of the Demazure module $V_w(\lambda)$. Kashiwara \cite{Kas93} extended the definition of Demazure crystals and gave an explicit construction for $\B_w(\lambda)$.

For $u,v\in W$ and $\mu,\nu\in P^+$, the subset $\B_u(\mu) \otimes \B_v(\nu)$ of $\B(\mu) \otimes \B(\nu)$ is not, in general, a direct sum of Demazure crystals. In part, this reflects the fact that the tensor product module $V_{u}(\mu) \otimes V_{v}(\nu)$ does not, in general, admit an excellent filtration. 
On the other hand, this reflects the fact that the tensor product crystal
$\B_u(\mu) \otimes \B_v(\nu)$ is not, in general, the crystal of $V_{u}(\mu) \otimes V_{v}(\nu)$. 

Lakshmibai, Littelmann and Magyar \cite{LLM02} and independently Joseph \cite{Jos03} proved $\B_{e}(\mu) \otimes \B_v(\nu)$ is a direct sum of Demazure crystals, giving a combinatorial version of Mathieu's result \cite{Mat89}.
There are other special cases of tensors of Demazure crystals studied using Kirillov--Reshetikhin crystals \cite{Naoi13, Naoi12,LS19,Gib21}.
Kouno \cite{Kou20} recently gave a complete characterization for when $\B_u(\mu) \otimes \B_v(\nu)$ is a direct sum of Demazure crystals. This characterization, however, considers certain combinatorial properties that depend on knowing exactly which Weyl group elements and dominant weights index the Demazure crystals being tensored.

In this paper, we replace Kouno's global condition with a local criterion that does not require knowledge of the combinatorial properties above and hence is more generally applicable. More specifically, we study a broader class of subsets of crystals which we call \emph{extremal}. These are subsets $X\subseteq \B(\lambda)$ characterized by the so-call \emph{string property}: for any $i$-string $S \subset \B(\lambda)$ with highest weight element $b$, the intersection $S \cap X$ is either $\varnothing$, $\{b\}$, or $S$. Kashiwara \cite{Kas93} notes extremality is one of the remarkable properties enjoyed by the Demazure crystals $\B_w(\lambda)$. Thus, while all Demazure crystals are extremal, not all extremal subsets are Demazure.

As with the Demazure property, the extremal property is not preserved by tensor products. Our first main result is a local characterization for when the tensor product of extremal subsets is extremal. Here the $e_i$ and $f_i$ are the raising and lowering operators, respectively, on the crystals.

\begin{theorem}\label{thm:extreme-hinge}
  For $X\subset \B(\lambda)$ and $Y\subset\B(\mu)$ extremal subsets, the tensor product $X \otimes Y$ is an extremal subset of $\B(\lambda) \otimes \B(\mu)$ if and only if for every $x \otimes y \in X \otimes Y$ for which $e_i(x\otimes y) = e_i(x) \otimes y \neq 0$ and $f_i(x\otimes y) = x \otimes f_i(y) \neq 0$, we have $f_i(y) \in Y$.
\end{theorem}

Any tensor product of Demazure crystals which decomposes as a direct sum of Demazure crystals is also extremal, since Demazure crystals are extremal. Our second main result is that, remarkably, the converse is true as well.

\begin{theorem}\label{thm:extreme-demazure}
  For $\lambda,\mu\in P^+$ and $w,u\in W$, we have $\B_w(\lambda) \otimes \B_u(\mu)$ is a direct sum of Demazure crystals if and only if it is extremal.
\end{theorem}

Theorems~\ref{thm:extreme-hinge} and \ref{thm:extreme-demazure} give a new, local characterization to determine when a tensor product retains its Demazure structure. We apply this characterization to show the connected component of $\B_w(\lambda)^{\otimes m}$ containing $b_{\lambda} \otimes \cdots \otimes b_{\lambda}$ is a direct sum of Demazure crystals, even when the full product is not.

We conclude with some evidence suggesting the existence of a modified definition of the tensor product rule for crystals which preserves the extremal property. Such a rule might lead to a characterization for which tensor products of Demazure modules admit excellent filtrations.

%
\section{Normal crystals}
%
\label{sec:normal-crystals}

Let $\g$ be the complex reductive Lie algebra associated with the Lie group $G$. Let $I$ be the set of vertices of the Dynkin diagram of $\g$, and let $P$ be the weight lattice of $\g$. For each $i\in I$, we have the simple root $\alpha_i\in P$ and its coroot $\alpha^{\vee}\in P^\vee= \Hom_{\mathbb{Z}}(P,\mathbb{Z})$. We review Kashiwara's theory of $\g$-crystal, restricting to the category of highest weight crystals. For further details, see \cite{Kas94}. 

A (normal) \emph{$\g$-crystal} $\B$ consists of a (finite) set $\B$ together with maps
\[
  \wt : \B \rightarrow P, \hspace{1em}
  \varepsilon_i,\varphi_i : \B \rightarrow \mathbb{Z}, \hspace{1em}
  e_i,f_i  : \B \rightarrow \B \sqcup \{0\}  
\]
subject to the following axioms for all $i\in I$ and all $b,b'\in \B$,
\begin{enumerate}
\item[(C1)] $\varphi_i(b) - \varepsilon_i(b) = \langle \alpha^{\vee}_i , \wt(b) \rangle$; 
\item[(C2)] if $e_i(b)\in \B$, then $\wt(e_i(b)) = \wt(b) + \alpha_i$;\\
  if $f_i(b)\in \B$, then $\wt(f_i(b)) = \wt(b) - \alpha_i$;
\item[(C3)] $b' = e_i(b)$ if and only if $b = f_i(b')$;
\item[(C4)] $\varepsilon_i(b) = \max\{k\ge 0 \mid e_i^k(b)\in \B\}$, and
  $\varphi_i(b) = \max\{k\ge 0 \mid f_i^k(b)\in \B\}$.
\end{enumerate}
The \emph{crystal operators} $\{e_i,f_i\}_{i\in I}$ are the $q=0$ limits of the Chevalley generators
of the quantum group $U_q(\g)$.

Let $P^+ = \{ \lambda\in P \mid \langle \alpha^{\vee}_i,\lambda \rangle \ge 0 \ \forall i\}$ be the set of \emph{dominant weights}. This set naturally indexes the finite dimensional irreducible integrable $U_q(\g)$-modules. For $\lambda\in P^{+}$, let $\B(\lambda)$ be the normal crystal associated with the crystal base of the simple $U_q(\g)$-module with highest weight $\lambda$.

Following \cite{Jos03} we introduce the monoids $\E$ and $\F$ generated by $\{e_i\}_{i\in I}$ and $\{f_i\}_{i\in I}$, respectively. 
For $b_{\lambda}$ the unique element of $\B(\lambda)$ with weight $\lambda$, we have
\[\E \{b_{\lambda}\} = 0, \mbox{ and }
  \F \{b_{\lambda}\} = \B(\lambda) \sqcup \{0\}.\]
In general, we say an element $b\in \B$ is \emph{highest weight} if $\E \{b\}=0$. 

A connected crystal $\B$ is a \emph{highest weight crystal with highest weight $\lambda$} if there exists a highest weight element $b_{\lambda}\in\B$ of weight $\lambda$ such that $\B = \F \{b_{\lambda}\}$.

Given two $\g$-crystals $\B,\B'$, a \emph{morphism} from $\B$ to $\B'$ is a map which commutes with the structure maps. Note not every highest weight crystal with highest weight $\lambda$ is isomorphic to $\B(\lambda)$. Henceforth, we consider only those $\g$-crystals $\B$ for which every connected component is isomorphic to $\B(\lambda)$ for some $\lambda\in P^+$.

%
\section{Demazure crystals}
%
\label{sec:excellent-crystals}

Let $W$ denote the Weyl group of $\g$ which, as a Coxeter group, is equipped with a \emph{length function} $\ell: W \rightarrow \mathbb{N}$ and \emph{Bruhat} partial order given by transitive closure of relations $u \prec u t$ for $t$ a reflection with $\ell(u) \le \ell(ut)$. Equivalently, $u\prec w$ if and only if a(ny) reduced expression for $w$ contains as a subword a reduced expression for $u$. We refer the reader to \cite{BB05} for classical results on Coxeter groups.

Littelmann \cite{Lit95} conjectured, and proved for classical types, the existence of a subset $\B_w(\lambda) \subseteq \B(\lambda)$ for any $w\in W$ whose character equals that of the Demazure module $V_w(\lambda)$ computed by Demazure's character formula \cite{Dem74,Jos85}. Kashiwara \cite{Kas93} generalized Littelmann's construction as follows.

For $w\in W$ and $s_{i_1} \cdots s_{i_{\ell}}$ a reduced expression for $w$, define
\begin{equation} \label{eq:F-definition}
  \F_w = \bigcup_{m_i \in \mathbb{N}} f_{i_1}^{m_1} f_{i_2}^{m_2} \cdots f_{i_{\ell}}^{m_{\ell}} \subset \F ,
\end{equation}
and define $\E_w$ similarly. 

Given any $X \subseteq B(\lambda)$ and $w \in W$, the set $\F_w(X)$ is independent of the choice of reduced expression for $w$ \cite{Jos03}. Hence, the following construction is well-defined.

\begin{definition}[\cite{Kas93}] \label{def:Dem}
  For $\lambda\in P^+$ and $w\in W$, the \emph{Demazure crystal}  $\B_w(\lambda)$ is
  \begin{equation} \label{eq:Dem-definition}
    \B_w(\lambda) = \F_{w} \{ b_\lambda \} .
  \end{equation}
\end{definition}
Notice that while $B_w(\lambda)$ is closed under $\E$ it is not closed under $\F$.

The following is a restatement of the fact that $\B_w(\lambda)$ is well-defined.

\begin{proposition}[\cite{Kas93}]  \label{prop:kas}

  For $\lambda\in P^+$ and $w\in W$, if $s_{i_1} \cdots s_{i_{\ell}}$ is any reduced expression for $w$, then for any $b \in \B_w(\lambda)$, there exist $m_j \ge 0$ such that
  \[ b = f_{i_1}^{m_1} f_{i_2}^{m_2} \cdots f_{i_{\ell}}^{m_{\ell}} (b_{\lambda}). \]
 \end{proposition}

The following converse will be useful in identifying Demazure crystals.

\begin{proposition}\label{prop:subword}
  For $\lambda\in P^+$, if $b = f_{i_1}^{m_1} f_{i_2}^{m_2} \cdots f_{i_{k}}^{m_{k}} (b_{\lambda})$ with $m_j \ge 0$ and $k$ minimal, then $s_{i_1} \cdots s_{i_{k}}$ is a reduced expression for some $w \in W$.
\end{proposition}

\begin{proof}
  We proceed by induction on $k$, noting the base case $k=1$ is trivial since $s_i$ is always reduced. Now assume the result for $k$, and let $b\in\B(\lambda)$ be written
  \[ b = f_{i}^{m} f_{i_1}^{m_1} \cdots f_{i_{k}}^{m_{k}} (b_{\lambda}) \]
  with $k+1$ minimal. Set $b_0 = f_{i_1}^{m_1} \cdots f_{i_{k}}^{m_{k}} (b_{\lambda})$. Note $k$ is minimal among all such expansions, since otherwise the expansion for $b$ could be shortened as well. Therefore, by induction, $s_{i_1} \cdots s_{i_{k}}$ is a reduced expression for some $u \in W$. In particular, $b_0 \in \F_{u}\{b_{\lambda}\}$ and $\ell(u)=k$.

  We claim $u \prec s_{i} u$. If not, then $s_i u \prec u$. Thus there exists a reduced expression for $u$, say $s_i s_{j_2} \cdots s_{j_{k}}$, and, by Proposition~\ref{prop:kas}, $b_0 = f_{i}^{n} f_{j_2}^{n_2} \cdots f_{j_{k}}^{n_{k}} (b_{\lambda})$. However, applying $f_{i}^{m}$ then gives a strictly shorter expansion
  \[ b = f_{i}^{m} (b_0) = f_{i}^{m+n} f_{j_2}^{n_2} \cdots f_{j_{k}}^{n_{k}} (b_{\lambda}), \]
  contradicting the minimality of $k+1$. Thus $u \prec s_{i} u$ as claimed. In particular, $s_i s_{i_1} \cdots s_{i_{k}}$ is a reduced expression for $s_i u$.  
\end{proof}

For $\lambda \in P^+$, let $W_\lambda$ be the stabilizer subgroup of $\lambda$ in $W$, and denote by $\floor{w}^\lambda$ and $\ceiling{w}^\lambda$ the minimal and maximal length coset representatives of $wW_\lambda$, respectively.
Note $\B_w(\lambda) = \B_{\floor{w}^\lambda}(\lambda) = \B_{\ceiling{w}^\lambda}(\lambda)$. 

\begin{proposition}\label{prop:minimal}
  Let $v,\, w\in W$ and $\lambda\in P^+$. Then $\B_v(\lambda) \subseteq \B_w(\lambda)$ if and only if $\floor{v}^\lambda\preceq w$ if and only if $v \preceq \ceiling{w}^\lambda$. In particular, $\B_v(\lambda) = \B_w(\lambda)$ only when $v \in w W_\lambda$.
\end{proposition}

\begin{proof} 
  If $\floor{v}^\lambda\preceq w$, then by \cite[Prop.~3.2.4]{Kas93}, we have $\B_v(\lambda) = \B_{\floor{v}^\lambda}(\lambda) \subseteq \B_w(\lambda)$. Similarly, if $v \preceq \ceiling{w}^\lambda$, the same result implies $\B_v(\lambda) \subseteq \B_w(\lambda) = \B_{\ceiling{w}^{\lambda}}(\lambda)$.

  Suppose $\B_v(\lambda) \subseteq \B_w(\lambda)$. Then $\floor{v}^\lambda(\lambda)$ is a weight that occurs in $\B_{\floor{v}^\lambda}(\lambda) = \B_v(\lambda) \subseteq \B_w(\lambda) = \B_{\ceiling{w}^{\lambda}}(\lambda)$, and so $\floor{v}^\lambda \preceq \floor{w}^\lambda \preceq w$ and $v \preceq \ceiling{v}^{\lambda} \preceq \ceiling{w}^{\lambda}$.
\end{proof}

%
\section{Extremal crystals}
%
\label{sec:extremal-crystals}

Kashiwara \cite{Kas93} showed $\B_w(\lambda)$ satisfies the following properties.
\begin{enumerate}
\item[(D1)] $\E \left( \B_w(\lambda) \right) \subset \B_w(\lambda) \sqcup \{0\}$;
\item[(D2)] if $s_i w \prec w$, then $\B_w(\lambda) = \{ f_{i}^{m} (b) \mid m \ge 0, b\in \B_{s_iw}(\lambda), e_i(b)=0\}\setminus\{0\}$; 
\item[(D3)] for any $i$-string $S$, $S \cap \B_w(\lambda)$ is either $\varnothing$ or $S$ or $\{b\}$, where $e_i(b)=0$.
\end{enumerate}
Here an \emph{$i$-string} is a connected subset of a crystal closed under both $\E_i$ and $\F_i$, where these denote the monoids generated by $e_i$ and $f_i$, respectively. 

Joseph \cite{Jos03} considered subsets of $\B(\lambda)$ satisfying (D1) and (D3). Following recent work of the extremal authors \cite{AG21}, we refer to such subsets as extremal. 

\begin{definition}
  A subset $X \subseteq \B(\lambda)$ is \emph{extremal} if $X$ is nonempty and for any $i$-string $S$ of $\B(\lambda)$, $S \cap X$ is either $\varnothing$ or $S$ or $\{b\}$, where $e_i(b)=0$.
  \label{def:extremal-crystal}
\end{definition}

Notice for $X$ extremal, we have $\E X \subset X \sqcup \{0\}$. In particular, since $X$ is nonempty, it must contain a highest weight element of $\B$. However, not all extremal subsets are Demazure crystals. For example, take $\g = A_2$ and $\lambda = \omega_1 + \omega_2$. Then $X = \{b_{\lambda}, f_1(b_{\lambda}), f_2(b_{\lambda})\}$ is extremal, but not Demazure.

%
\section{Tensor products of normal crystals}
%
\label{sec:tensor}

The \emph{direct sum} $\B_1 \oplus \B_2$ of two $\g$-crystals is their disjoint union with the obvious maps. Thus every crystal decomposes as a direct sum of highest weight crystals.

The \emph{tensor product} $\B_1 \otimes \B_2$ is the set $\{ b_1 \otimes b_2 \mid b_1\in\B_1 \mbox{ and } b_2\in\B_2 \}$ with maps
\begin{align*}
  \wt(b_1 \otimes b_2) & = \wt(b_1) + \wt(b_2), \\
  \varepsilon_i(b_1 \otimes b_2) & = \max( \varepsilon_i(b_1), \varepsilon_i(b_2) - \wt_i(b_1) ),\\
  \varphi_i(b_1 \otimes b_2) & = \max( \varphi_i(b_2), \varphi_i(b_1) + \wt_i(b_2) ),
\end{align*}
where $\wt_i(b) = \langle \alpha^{\vee}_i , \wt(b) \rangle$. The crystal operators $e_i, f_i$ are defined by
\begin{align*}
  e_i(b_1 \otimes b_2) &=
  \begin{cases}
    e_i(b_1) \otimes b_2 & \mbox{if } \varepsilon_i(b_2) \le \varphi_i(b_1), \\
    b_1 \otimes e_i(b_2) & \mbox{if } \varepsilon_i(b_2) > \varphi_i(b_1);
  \end{cases}\\
  f_i(b_1 \otimes b_2) &= 
  \begin{cases}
    f_i(b_1) \otimes b_2 & \mbox{if } \varepsilon_i(b_2) < \varphi_i(b_1), \\
    b_1 \otimes f_i(b_2) & \mbox{if } \varepsilon_i(b_2) \geq \varphi_i(b_1).
  \end{cases}
\end{align*}
The tensor product is associative but not commutative, though it is functorial.

\begin{theorem}[\cite{Kas91}]
  For $\lambda,\mu\in P^+$, $\B(\lambda) \otimes \B(\mu)$ is a crystal for $V(\lambda)\otimes V(\mu)$.
  \label{thm:func}
\end{theorem}

%
\section{Tensor products of Demazure crystals}
%
\label{sec:tensor-dem}

Given Demazure crystals $\B_w(\lambda),\B_u(\mu)$, it is not the case that $\B_w(\lambda)\otimes\B_u(\mu)$ is always a direct sum of Demazure crystals. For example, when $\g = A_2$, the product $\B_{s_2}(\omega_2)\otimes\B_{s_1}(\omega_1)$ is not direct sum of Demazure crystals, which is expected since the module $V_{s_2}(\omega_2)\otimes V_{s_1}(\omega_1)$ does not admit an excellent filtration. 

The issue here is more fundamental, however, as there is no analog of Theorem~\ref{thm:func} for Demazure crystals. For example, take $\g = A_2$ with $\lambda = \omega_1 + \omega_2$ and $w = s_1 s_2$. Then $V_{w}(\lambda)\otimes V_{w}(\lambda)$ admits an excellent filtration, but $\B_{w}(\lambda)\otimes \B_{w}(\lambda)$ is not a direct sum of Demazure crystals and so is not a crystal for $V_{w}(\lambda)\otimes V_{w}(\lambda)$. 

Kouno \cite{Kou20} characterized when $\B_u(\mu) \otimes \B_v(\nu)$ is a direct sum of Demazure crystals. For any $\sigma \in W$ let  $W_\sigma \subseteq W$ denote the parabolic subgroup. Note, despite similar notation, $W_\sigma$ is not related to the stabilizer subgroup $W_\lambda$ for $\lambda \in P^+$.
\[W_\sigma = \langle \; s_i \in W \; | \; s_i \sigma \prec \sigma \; \rangle. \]
  
\begin{theorem}[\cite{Kou20}] \label{thm:kouno}
  Let $\lambda, \mu \in P^+$ and $u, w \in W$. Then $B_w(\lambda) \otimes B_u(\mu)$ is a direct sum of Demazure crystals if and only if $\floor{w}^\lambda \in W_{\ceiling{u}^\mu}$.
\end{theorem}
 
Kouno's proof is quite technical, making detailed use Lakshmibai-Seshadri paths for crystals \cite{LLM02, Lit95-2}. One of Kouno's main applications is to the key positivity problem, where he notes in \cite[Thm~8.2]{Kou20} the following special case of Theorem~\ref{thm:kouno} for which we give a direct proof.

\begin{proposition}
  For $\lambda,\mu\in P^+$ and $w\in W$, we have $\B_w(\lambda) \otimes \B(\mu)$ is a direct sum of Demazure crystals.
  \label{prop:excellent-tensor}
\end{proposition}

\begin{proof}
  By definition, $\B_w(\lambda)$ admits a filtration by Demazure crystals,
  \[
  \{b_\lambda\} \subset X_{i_1}\subset \dots \subset X_{i_{\ell-1}} \subset X_{i_\ell}=\B_w(\lambda)
  \] 
  where $s_{i_\ell} \dots s_{i_1}$ is a reduced expression for $w$ and $\F_{i_r}(X_{i_{r-1}})=X_{i_r}$.  Consequently,  we have a filtration on $X \otimes \B(\mu)$ with $k$ minimal such that
  \[
  \{b_\lambda\} \otimes \B(\mu) \subset X_{i_1}\otimes \B(\mu)\subset \dots \subset X_{i_{k-1}}\otimes \B(\mu) \subset X_{i_k}\otimes \B(\mu)=X\otimes \B(\mu). 
  \] 
  By \cite[Thm.~2.11]{Jos03}, $\{b_\lambda\} \otimes \B(\mu)$ is a direct sum of Demazure crystals.  So let $X_{i_{k-1}}= \B_w(\lambda)$ and $i=i_k$ so that $X = \B_{s_iw}(\lambda)$.  For induction, assume $\B_w(\lambda)\otimes \B(\mu)$ is a direct sum of Demazure crystals, and let $x\otimes b \in X \otimes \B(\mu)$ such that $x \not\in \B_w(\lambda)$ but $e_i^m(x) \in \B_w(\lambda)$ for some $m>0$.  If $e_i(x \otimes b) =0$ then necessarily $\varepsilon_i(x)=0$, a contradiction. Thus we have two possibilities.  If $e_i(x \otimes b) = e_i(x) \otimes b$ then $\varepsilon_i(b)\leq \varphi_i(x)$.  However $\varphi_i(x)<\varphi_i(e_i^m(x))$ and thus $e_i^m(x \otimes b) = e_i^m(x) \otimes b \in \B_w(\lambda) \otimes \B(\mu)$.  If instead $e_i(x \otimes b) = x \otimes e_i(b)$ then $\varepsilon_i(b)> \varphi_i(x)$.  So let $t = \varepsilon_i(b) - \varphi_i(x)$ so that $\varepsilon_i(e_i^t(b))= \varphi_i(x)$ and $e_i^t(x \otimes b) = x \otimes e_i^t(b)$.  From this and the fact that $\varepsilon_i(x)>0$ we see that $e_i(x \otimes e_i^t(b)) = e_i(x) \otimes e_i^t(b)$.  Iterating as in the previous case we obtain $e_i^{m+t}(x \otimes b) = e_i^m(x \otimes e_i^t(b)) = e_i^m(x) \otimes e_i^t(b)$ and thus $X \otimes \B(\mu) \subset \F_i(\B_w(\lambda) \otimes \B(\mu))$.  The reverse inclusion follows from the fact that $f_i^k(x \otimes b)$ will always be contained in $X \otimes \B(\lambda)$ for any $x \otimes b \in \B_w(\lambda) \otimes \B(\mu)$ and any $k\geq 0$.
\end{proof}

In order to apply Kouno's theorem, it is necessary to have the exact values of $\lambda, \mu \in P^+$ and $w,u \in W$ that index the Demazure crystals. Sometimes, this information is not readily available. It is possible to know the crystals in question are Demazure without knowing exactly which Demazure crystals they are. In the following section we remove this constraint by using extremal crystals to give an alternative local characterization for when a product of Demazure crystals is a direct sum of Demazure crystals. 

%
\section{Tensor products of extremal crystals}
%
\label{sec:tensor-extreme}

As with Demazure crystals, tensor products of extremal crystals are not always extremal. For example, take $\g = A_1$, then $\B_{s_1}(\omega_1) \otimes \{b_{\omega_1}\}$ is $2$-dimensional with unique highest weight $2 \omega_1$, and so is neither extremal nor Demazure. Below, we characterize when the tensor product of extremal crystals remains extremal.

\begin{definition}
  An element $x \otimes y \in \B(\lambda)\otimes\B(\mu)$ is called an \emph{$i$-hinge} if
  \begin{enumerate}
  \item $\varepsilon_i(x) > 0$, $\varphi_i(x) = 0$,
  \item $\varepsilon_i(y) = 0$, $\varphi_i(y) > 0$.
  \end{enumerate}
\end{definition}

In terms of the crystal structure, an $i$-hinge is an element where $e_i$ acts on the left factor and $f_i$ acts on that right factor. 

\begin{proposition}\label{prop:leftright}
  For $x \otimes y \in \B(\lambda)\otimes\B(\mu)$ an $i$-hinge, we have
  \begin{itemize}
  \item $e_i(x\otimes y) = e_i(x) \otimes y \in \B(\lambda)\otimes\B(\mu)$, and
  \item $f_i(x\otimes y) = x \otimes f_i(y) \in \B(\lambda)\otimes\B(\mu)$.
  \end{itemize}
\end{proposition}

\begin{proof}
  Since $\varphi_i(x) = 0 = \varepsilon_i(y)$, the tensor product rules give $e_i(x\otimes y) = e_i(x) \otimes y$ and $f_i(x\otimes y) = x \otimes f_i(y)$. Since $\varepsilon_i(x)>0$, we have $e_i(x) \in \B(\lambda)$, and since $\varphi_i(y)>0$, we have $f_i(y)\in\B(\mu)$. 
\end{proof}

Thus, we say an $i$-hinge $x \otimes y \in X\otimes Y$ is \emph{broken} if $f_i(y) \not\in Y$. Determining whether $X \otimes Y$ contains a broken $i$-hinge is a local, easily checked property which precisely characterizes when tensor products of extremal crystals are extremal.

\begin{theorem}\label{thm:badguy}
  Let $X\subset \B(\lambda)$ and $Y\subset\B(\mu)$ be extremal subsets. Then $X \otimes Y$ is an extremal subset of $\B(\lambda) \otimes \B(\mu)$ if and only if $X\otimes Y$ does not contain a broken $i$-hinge for any $i$.
\end{theorem}

\begin{proof}
  Suppose $x \otimes y \in X \otimes Y$ is an $i$-hinge and $f_i(y) \not\in Y$. By Proposition~\ref{prop:leftright}, so $e_i(x\otimes y) = e_i(x) \otimes y \in X \otimes Y$, since $X$ is extremal. Thus $x \otimes y$ is not the top of its $i$-string. Similarly, $f_i(x\otimes y) = x \otimes f_i(y) \in X \otimes \B(\mu)$. However, since $f_i(y) \not\in Y$, we have $x \otimes f_i(y) \not\in X \otimes Y$. Thus the $i$-string through $x \otimes y$ violates the extremal criteria for $X \otimes Y$.

  Now suppose $X \otimes Y$ is not extremal. Since both $X$ and $Y$ are extremal, we must have $\E (X\otimes Y) \subset X \otimes Y \sqcup \{0\}$ since $\E X \subset X \sqcup \{0\}$ and $\E Y \subset Y \sqcup \{0\}$. Thus there exists $x_0 \otimes y_0 \in X \otimes Y$ and $i\in I$ such that
  \begin{itemize}
  \item $\varepsilon_i(x_0 \otimes y_0)=0$,
  \item $f_i(x_0 \otimes y_0) \in X \otimes Y$, and
  \item for some $m>1$, we have $f_i^m(x_0 \otimes y_0) \not\in X \otimes Y \sqcup \{0\}$.
  \end{itemize}
  Taking $m>1$ above to be minimal, define $x \otimes y = f_i^{m-1}(x_0 \otimes y_0) \in X \otimes Y$. We will show $x \otimes y$ is an $i$-hinge with $f_i(y) \not\in Y$.

  We claim $e_i(x \otimes y) = e_i(x) \otimes y$. Indeed, if $e_i(x \otimes y) = x \otimes e_i(y)$, then $\varepsilon_i(y) > \varphi_i(x)$, and so $f_i(x \otimes y) = x \otimes f_i(y)$ as well. But since $\varepsilon_i(y) > \varphi_i(x) \ge 0$ and $Y$ is extremal, $f_i(y) \in Y \sqcup \{0\}$, contradicting that $f_i(x \otimes y) \not\in X\otimes Y \sqcup \{0\}$. Thus $e_i(x \otimes y) = e_i(x) \otimes y$ and, in particular, $\varepsilon_i(y) \le \varphi_i(x)$. Moreover, since $m>1$, we have $e_i(x) \otimes y = e_i(x \otimes y) \neq 0$, and so $\varepsilon_i(x) > 0$.
  
  By the crystal axioms, it follows as well that $e_i^k(x \otimes y) = e_i^k(x) \otimes y$ for $k < m$. In particular, $y = y_0$, and so $\varepsilon_i(y) = 0$.

  Furthermore, since $X$ is extremal and $e_i(x),x\in X$, we must have $f_i(x) \in X \sqcup \{0\}$ and so $f_i(x) \otimes y \in X\otimes Y \sqcup \{0\}$. Thus we must have $f_i(x \otimes y) = x \otimes f_i(y)$. In particular, $\varepsilon_i(y) \ge \varphi_i(x)$. By the previous inequality, we have $\varphi_i(x) = \varepsilon_i(y) = 0$.
  
  Finally, since $x \otimes f_i(y) = f_i(x \otimes y) \not\in X \otimes Y \sqcup \{0\}$, we have $f_i(y) \not\in Y \sqcup \{0\}$ and so $\varphi_i(y) > 0$. Therefore $x \otimes y$ is an $i$-hinge with $f_i(y) \not\in Y$ as claimed.
\end{proof}

Thus Theorem~\ref{thm:extreme-hinge} follows from Theorem~\ref{thm:badguy} and Proposition~\ref{prop:leftright}.
Notice, if $X\subset \B(\lambda)$ has only the highest weight element or if $Y \subset \B(\mu)$ contains all possible elements, then $X \otimes Y$ contains no hinges. That is, both $\{b_\lambda\}\otimes\B_u(\mu)$ and $\B_w(\lambda) \otimes \B(\mu)$ are extremal subsets of $\B(\lambda) \otimes \B(\mu)$. Of course, this follows by Theorem~\ref{thm:kouno} since both are direct sums of Demazure crystals.

%
\section{Extremal tensor products}
%
\label{sec:main}

Our interest in extremal crystals lies not exclusively in their implicit structure, but rather in the structure imposed on the factors when a tensor product of two subsets of crystals is extremal. We begin with the following.

\begin{proposition}\label{prop:extreme}
  If $X \otimes Y \subset \B(\lambda) \otimes \B(\mu)$ is an extremal subset, then $\E (X) \subset X \sqcup \{0\}$. Furthermore, if $\E (Y) \subset Y \sqcup \{0\}$, then $X \subset \B(\lambda)$ is an extremal subset.
\end{proposition} 

\begin{proof}
  Suppose $X \otimes Y \subset \B(\lambda) \otimes \B(\mu)$ is an extremal subset. Take $y\in Y$ such that $e_i(y) \not\in Y$. We claim $e_i(x \otimes y) = e_i(x) \otimes y$ for any $x\in X$. In particular, when $X \otimes Y$ is extremal, this implies $e_i(x)\in X$. To see the claim, notice if $e_i(y) = 0$, then $\varepsilon_i(y) = 0$ and so $\varphi_i(x) \ge \varepsilon_i(y)$ showing $e_i(x \otimes y) = e_i(x) \otimes y$. Alternatively, if $e_i(y) \neq 0$ and $e_i(x \otimes y) = x \otimes e_i(y)$, then since $X \otimes Y$ is extremal, this implies $e_i(y)\in Y$ contradicting the choice of $y$. Thus $\E (X) \subset X \sqcup \{0\}$.

  Now suppose, in addition, that $\E (Y) \subset Y \sqcup \{0\}$. Then we may take $y\in Y$ such that $e_i(y) =0$. Consider $x\in X$ for which $e_i(x), f_i(x) \neq 0$. By the prior argument, $e_i(x)\in X$, and so $e_i( x \otimes y ) = e_i(x) \otimes y \neq 0$ since $\varphi_i(y)=0$. We must show $f_i(x)\in X$. Since $X \otimes Y$ is extremal and $e_i( x \otimes y )\neq 0$, we have $f_i(x \otimes y) \in X \otimes Y \sqcup \{0\}$. Since $f_i(x) \neq 0$, we have $\varphi_i(x)> 0 = \varepsilon_i(y)$, and so $f_i(x \otimes y) = f_i(x) \otimes y$, which ensure $f_i(x) \in X$. Thus $X$ is extremal.  
\end{proof}

Notice that in Proposition~\ref{prop:extreme}, $Y$ need not be extremal, even when $\E Y \subset Y \sqcup \{0\}$. For example, for $\g = A_2$ with $\lambda = 2 \omega_1 + 2\omega_2$ and $\mu = 2 \omega_1$, take $X = \{b_{\lambda}\}$ and $Y = \{ b_{\mu}, f_1(b_{\mu}) \}$. Then $X \otimes Y \cong \B_{\mathrm{id}}(\lambda+\mu) \oplus \B_{\mathrm{id}}(\lambda+\omega_2)$ is a direct sum of Demazure crystals, and hence extremal, but $Y$ is not even extremal.

Any subset $\B_w(\lambda) \otimes \B_u(\mu) \subset \B(\lambda) \otimes \B(\mu)$ which is a direct sum of Demazure crystals is also an extremal subset, since Demazure crystals are extremal. Amazingly, the converse is also true. 

\begin{theorem}
  For $\lambda,\mu\in P^+$ and $w,u\in W$, we have $\B_w(\lambda) \otimes \B_u(\mu)$ is an extremal subset of $\B(\lambda) \otimes \B(\mu)$ if and only if $\floor{w}^\lambda \in W_{\ceiling{u}^\mu}$.
\end{theorem}

\begin{proof}
Suppose that $\floor{w}^\lambda \not\in W_{\ceiling{u}^\mu}$. Then there exists $s_i \prec \floor{w}^\lambda$ such that $s_i \ceiling{u}^\mu \succ \ceiling{u}^\mu$. Since $\ceiling{u}^\mu$ is maximal length in $uW_\mu$ then $s_i \ceiling{u}^\mu (\mu) \neq \ceiling{u}^\mu(\mu)$. Thus $B_u(\mu) = B_{\ceiling{u}^\mu}(\mu) \subsetneqq  B_{s_i\ceiling{u}^\mu}(\mu) = B_{s_iu}(\mu)$ and we can find $y \in B_u(\mu)$ for which $\varphi_i(y)>0$ but $f_i(y) \not\in B_u(\mu)$.  On the other hand, since $s_i \prec \floor{w}^\lambda$ and $\floor{w}^\lambda$ is the shortest coset representative, by Proposition~\ref{prop:kas}, there exists an element $x \in B_w(\lambda)$ given by $x = f_i^m f_{j_k}^{m_k} \dots f_{j_1}^{m_1}(b_\lambda)$ with all $m's$ maximal for which $\varepsilon_i(x) >0$ and $\varphi_i(x) =0$. So if we consider $x \otimes y \in B(\lambda) \otimes B(\mu)$ we can see that $e_i^m(x \otimes y) = e_i^m(x) \otimes y \in B_w(\lambda) \otimes B_u(\mu)$ but $f_i(x \otimes y) = x \otimes f_i(y) \not\in B_w(\lambda) \otimes B_u(\mu)$. Thus $B_w(\lambda) \otimes B_u(\mu)$ contains a broken $i$-hinge and by Theorem \ref{thm:badguy} cannot be extremal. 

Now suppose $B_w(\lambda) \otimes B_u(\mu)$ is not an extremal subset. Then by Theorem \ref{thm:badguy} there exists a broken $i$-hinge $x\otimes y \in B_w(\lambda) \otimes B_u(\mu)$. In particular, $0 \neq f_i(y) \not\in B_u(\mu)$ so $B_{\ceiling{u}^\mu}(\mu) = B_u(\mu) \subsetneqq B_{s_iu}(\mu)= B_{s_i\ceiling{u}^\mu}(\mu) $ and hence $s_i\ceiling{u}^\mu \not\in uW_\mu$. Since $\ceiling{u}^\mu$ has maximal length then $s_i\ceiling{u}^\mu \succ \ceiling{u}^\mu$, thus $s_i \not\in W_{\ceiling{u}^\mu}$. Moreover, since $x \in B_w(\lambda)$ with $\varepsilon_i(x) >0$ then $s_i \prec \floor{w}^\lambda$ which implies $\floor{w}^\lambda \not\in W_{\ceiling{u}^\mu}$.
\end{proof}

Combining this with Theorem \ref{thm:kouno} yields Theorem~\ref{thm:extreme-demazure}, giving a local characterization for which tensor products of Demazure crystals remain Demazure.

Notice Theorem~\ref{thm:extreme-demazure} is false for tensor products of extremal subsets. For example, take $\g = A_2$ with $\lambda = \omega_3$ and $\mu = \omega_1 + \omega_2$. Then both $X = \{b_{\lambda}\}$ and $Y = \{ b_{\mu}, f_1(b_{\mu}), f_2(b_{\mu}) \}$ are extremal, but $X \otimes Y \cong Y$ is extremal but not Demazure.

%
\section{Application to tensor squares}
%
\label{sec:squares}

Even when $\B_{w}(\lambda) \otimes \B_{u}(\mu)$ is not a direct sum of Demazure crystals, some connected components of it may be. For example, take $\g = A_2$ with $\lambda = \omega_1 + \omega_2$ and $w = s_1 s_2$. Then $\B_{w}(\lambda)\otimes \B_{w}(\lambda)$ has four connected components, two of which are Demazure crystals and two of which are not even extremal. 

Generalizing this example, we show the connected component of the tensor square $\B_{w}(\lambda) \otimes \B_{w}(\lambda)$ containing $b_{\lambda}\otimes b_{\lambda}$ is always a Demazure crystal, even when $\B_{w}(\lambda) \otimes \B_{w}(\lambda)$ is not a direct sum of Demazure crystals. 

\begin{lemma}\label{lem:diagonal}
  For $\lambda,\mu \in P^+$ such that $\langle\alpha_i^\vee,\mu\rangle = 0$ whenever $\langle\alpha_i^\vee,\lambda\rangle=0$, if $x \otimes y \in \F(\{b_{\lambda} \otimes b_{\mu}\}) \subseteq \B(\lambda) \otimes \B(\mu)$ and $x \in \B_w(\lambda)$ for some $w\in W$, then $y \in \B_w(\mu)$.
\end{lemma}

\begin{proof}
  We claim if $b_{\lambda} \otimes y \in \F(\{b_{\lambda} \otimes b_{\mu}\})$, then $y = b_{\mu}$. Since $b_{\mu}$ is highest weight, $\varepsilon_i(b_{\mu})=0$. Therefore, if $\varphi_i(b_{\lambda}) > 0$, then $f_i(b_{\lambda} \otimes b_{\mu}) = f_i(b_{\lambda}) \otimes b_{\mu}$; otherwise $\varphi_i(b_{\lambda}) = 0 = \varphi_i(b_{\mu})$, and so $f_i(b_{\lambda} \otimes b_{\lambda}) = b_{\lambda} \otimes f_i(b_{\mu}) = 0$. The claim follows.

  Now consider $x \otimes y \in \F(\{b_{\lambda} \otimes b_{\mu}\})$ with $x \in \B_w(\lambda)$. By Proposition~\ref{prop:kas}, there exist $i_1,\ldots,i_k$ and $m_1,\ldots,m_k \geq 0$ such that $x \otimes y = f_{i_1}^{m_1} \cdots f_{i_k}^{m_k} (b_{\lambda}) \otimes y$, where $s_{i_1} \cdots s_{i_{k}}$ is a reduced expression for $w$. If $\varepsilon_i(b_2) = 0$, then $e_i(b_1 \otimes b_2) = e_i(b_1) \otimes b_2$. Therefore there exist integers $n_1,\ldots,n_k$ with  $n_i \ge m_i \ge 0$ such that
  \[ e_{i_k}^{n_k} \cdots e_{i_1}^{n_1} (x \otimes y) = b_{\lambda} \otimes y_0 \]
  for some $y_0 \in \E_w\{y\}$. By the opening claim, this means $y_0 = b_{\mu}$. Reversing the relationship $y_0 \in \E_w\{y\}$ gives $y \in \F_w\{y_0\} = \F_w\{b_{\mu}\} = \B_w(\mu)$.
\end{proof}

Using this and Theorem~\ref{thm:extreme-demazure}, we have the following application.

\begin{theorem}\label{thm:square}
  For $\lambda,\mu \in P^+$ such that $\langle\alpha_i^\vee,\mu\rangle = 0$ whenever $\langle\alpha_i^\vee,\lambda\rangle=0$, and for $u,v\in W$ such that $u \preceq v$, we have
  \[ \F \left(\{b_{\lambda}\otimes b_{\mu}\right\}) \cap \B_{u}(\lambda) \otimes \B_{v}(\mu) \cong \B_{u}(\lambda+\mu). \]
\end{theorem}

\begin{proof}
  Suppose $x\otimes y\in\F(\{b_\lambda\otimes b_\mu\})$ is an $i$-hinge. Let $w$ be minimal such that $x \in \B_w(\lambda)$. Then $w \le u$ by Proposition~\ref{prop:minimal}. Furthermore, since $\varepsilon_i(x)>0$, by Proposition~\ref{prop:subword}, we have $\F_i(\B_w(\lambda)) = \B_w(\lambda)$. By Lemma~\ref{lem:diagonal}, $y \in \B_w(\mu)$. Therefore $f_i(y) \in \F_i(\B_w(\mu)) = \B_w(\mu) \subseteq \B_v(\mu)$ by Proposition~\ref{prop:minimal} and transitivity of $w \preceq u \preceq v$. In particular $x \otimes y$ cannot be a broken $i$-hinge in $B_u(\lambda)\otimes B_v(\mu)$. Thus it is extremal, and so by Theorem~\ref{thm:extreme-demazure}, it is Demazure. The result follows from that fact that $\F_u \left(\{b_{\lambda}\otimes b_{\mu}\}\right)$ is isomorphic to $\B_{u}(\lambda+\mu)$ \cite{Jos03}.
\end{proof}

The associativity of the tensor product gives the following consequence. This is also studied in \cite{DEKM} in the special case where $w$ is miniscule.

\begin{corollary}\label{cor:square}
  For $\lambda \in P^+$ and $w\in W$, for any $m \ge 1$ we have
  \[\F (\{b_{\lambda}\otimes \cdots \otimes b_{\lambda}\}) \cap \B_w(\lambda)^{\otimes m} \cong \B_{w}(m\lambda) .\]
\end{corollary}

%
\section{Concluding remarks}
%
\label{sec:another}

Interestingly, in the example above for $\g = A_2$ with $\lambda = \omega_1 + \omega_2$ and $w = s_1 s_2$, even though the tensor product of Demazure crystals $\B_{w}(\lambda)\otimes \B_{w}(\lambda)$ is not a direct sum of Demazure crystals, the corresponding tensor product of Demazure modules $V_{w}(\lambda)\otimes V_{w}(\lambda)$ does admit an excellent filtration. Indeed, if one removes from $\B_{w}(\lambda)\otimes \B_{w}(\lambda)$ the two edges ending in broken $i$-hinges, namely the $f_2$ edges from $b_{\lambda} \otimes f_1(b_{\lambda})$ and $b_{\lambda} \otimes f_1^2 f_2 (b_{\lambda})$, then the resulting structure becomes extremal and, moreover, is isomorphic to a direct sum of Demazure crystals.

As Kouno notes, Proposition~\ref{prop:excellent-tensor} implies the positivity for type A Demazure characters proven combinatorially by Haglund, Luoto, Mason and van Willigenberg \cite{HLMvW11}. Recently, Assaf \cite[Thm~5.2.2]{Assaf} gave a larger class of type A Demazure modules for which the product of the characters expands nonnegatively into Demazure characters. Yet this positivity does not follow from Theorem~\ref{thm:kouno}.

In fact, in many instances, though the characters exhibit the necessary positivity and the tensor product of the modules admits an excellent filtration, the tensor product of Demazure crystals does not decompose into a direct sum of Demazure crystals. This suggests Kashiwara's tensor product for crystals is not the correct tensor rule for Demazure crystals. This leads naturally to question the existence of a new and different tensor product for crystals that preserves the extremal property of subsets. Such a rule might lead to a characterization of when tensor products of Demazure modules admit excellent filtrations, fully generalizing the conjecture of Polo proved by Mathieu and solving the key positivity problem.

%
%

\bibliographystyle{plain} 
\bibliography{demazure}

\end{document}